\documentclass{amsart}
\usepackage[margin=1.2in]{geometry}

\usepackage{amsmath,amssymb,amsthm,amscd,mathrsfs}
\usepackage{mathtools}
\usepackage[all,2cell]{xy}
\UseAllTwocells

\usepackage[usenames,dvipsnames,svgnames,table]{xcolor}
\usepackage[linkbordercolor=BrickRed, citebordercolor=OliveGreen]{hyperref}
\usepackage{amsrefs}
\usepackage{marginnote}
\usepackage[british]{babel}

\usepackage{ulem}

\usepackage{tikz}
\usepackage{tikz-cd}

\newtheorem{theorem}{Theorem}[section]

\newtheorem{lemma}[theorem]{Lemma}
\newtheorem{setup}[theorem]{Setup}
\newtheorem{proposition}[theorem]{Proposition}
\newtheorem{corollary}[theorem]{Corollary}

\theoremstyle{definition}
\newtheorem{definition}[theorem]{Definition}

\newtheorem{remark}[theorem]{Remark}

\newtheorem{example}[theorem]{Example}

\DeclareMathOperator{\Hom}{Hom}

\DeclareMathOperator{\rad}{rad}

\DeclareMathOperator{\End}{End}

\begin{document}
\emergencystretch 3em

\title{Triangular decompositions: Reedy algebras and quasi-hereditary algebras}

 \address{Teresa Conde\\
   Faculty of Mathematics and TRR 358,
   University of Bielefeld \\
   Universit\"atsstra{\ss}e 25, 33615 Bielefeld, Germany } 
   \email{tconde@math.uni-bielefeld.de}

 \address{Georgios Dalezios\\
 Dipartimento di informatica, 
 Universit\'a degli Studi di Verona \\ Strada Le Grazie 15, 37134 Verona, Italia} 
 \email{georgios.dalezios@univr.it}

 \address{Steffen Koenig\\
 Institute of Algebra and Number Theory,
 University of Stuttgart \\ Pfaffenwaldring 57 \\ 70569 Stuttgart,
 Germany} \email{skoenig@mathematik.uni-stuttgart.de}

 \thanks{TC acknowledges support by the Deutsche Forschungsgemeinschaft (DFG, German Research Foundation) -- Project-ID 491392403 – TRR 358. GD acknowledges support by the project LAVIE -- Large views of small phenomena: decompositions, localizations, and representation type, FIS 00001706, funded by Program FIS2021 of Italian Ministry of University and Research.}

\author{Teresa Conde, Georgios Dalezios and Steffen Koenig}

\begin{abstract}
  Finite-dimensional Reedy algebras
  form a ring-theoretic analogue of Reedy
  categories and were recently proved to be quasi-hereditary. 
  We identify Reedy algebras with quasi-hereditary algebras admitting a
  triangular (or Poincar\'e--Birkhoff--Witt type) decomposition into the tensor product of two
  oppositely directed subalgebras over a common
  semisimple subalgebra. This exhibits homological and
  representation-theoretic structure of the
  ingredients of the Reedy decomposition and it allows to give a 
  characterisation of Reedy algebras in terms of idempotent ideals
  occurring in heredity chains, providing an analogue for Reedy algebras of a
  result of Dlab and Ringel on quasi-hereditary algebras.
\end{abstract}

\subjclass[2010]{16W70, 16G10 (Primary); 18N40 (Secondary)}
\keywords{Quasi-hereditary algebra, Reedy algebra, triangular decomposition}

\maketitle

\section{Introduction}
Finite-dimensional
Reedy algebras were recently introduced in \cite{DS} and
shown to be quasi-hereditary, in the sense of \cites{CPS,DR}. They form a ring-theoretic analogue of Reedy categories which classically originated
in homotopy theory as a generalisation of the cosimplicial
indexing category $\mathbf{\Delta}$ (see \cite[Definition~15.1.8]{Hir}) of finite
ordinals and weakly monotone functions between them. A fundamental property
of $\mathbf{\Delta}$ is
that every morphism therein factorises uniquely as the composite of a morphism
that weakly lowers the degree followed by a morphism that weakly raises the degree. This gives rise to the concept of Reedy categories. Reedy categories are crucial since categories of diagrams in a model category indexed by a Reedy category inherit a model category structure. A classical example is the Reedy model structure on cosimplicial or simplicial spaces, that is, on the category of covariant or contravariant functors from $\mathbf{\Delta}$ to the category of topological spaces (or to the category of simplicial sets) endowed with the Quillen model structure (see, for instance \cite{Bousfield-Kan}).

In order to pass to a ring-theoretic structure, the modification of the concept
of Reedy categories needed is twofold. Namely, one needs to consider versions
of Reedy categories which are linear (enriched over a ring or a field) and
also consist of finitely many objects. As a prototypical example, fix a field
$k$ and consider a truncation $\mathbf{\Delta}_{\leqslant n}$ of the
cosimplicial category at some non-negative integer $n$.
Then the path algebra $k\mathbf{\Delta}_{\leqslant n}$ modulo the two-sided ideal
generated by the cosimplicial relations that hold true up to level $n$, is a
finite-dimensional
Reedy algebra in the sense of \cite{DS} and thus a
quasi-hereditary algebra. Many more examples of Reedy algebras are provided in
this note. 

By definition, a Reedy algebra $A$ contains directed subalgebras
$A^+$ and $A^-$ satisfying an isomorphism called Reedy decomposition
(see Definition~\ref{def:Reedy}). In Theorem \ref{thm:characterization} we
prove that Reedy triples $(A,A^+,A^-)$ coincide
with triples $(A,C,B)$ of quasi-hereditary algebras, where $C$ and $B$
are oppositely directed subalgebras of $A$ satisfying an isomorphism of $C$--$B$--bimodules
$C\otimes_{S}B \cong A$ for $S=C\cap B$. The algebras $C$ and $B$ carry particular structures: the standard modules over $A$ restrict to
projective $C$--modules, while $B$ is coming
with an exact induction functor to $A$--modules that sends simple
$B$--modules to standard $A$--modules.
Such algebras have actually been introduced and studied in the
past in \cites{Koe1,Koe2,Koe95,Koe96}, where $B$ is called an exact Borel
subalgebra, $C$ is called a Delta subalgebra and the
Reedy decomposition is viewed as a triangular
decomposition that has been called
a Cartan decomposition. Such a triangular
decomposition is desirable to
have for algebras arising in algebraic Lie theory, but it is hard to
establish. The connection with Reedy algebras now provides plenty of examples
of such quasi-hereditary algebras, while it conversely provides ring-theoretic and homological structure to Reedy decompositions and the
subalgebras occurring therein. Motivated by this connection, it is proved in \cite[Theorem 3.7]{Rasmussen} that every quasi-hereditary monomial algebra has a Reedy decomposition.

Theorem \ref{thm:characterization} invites us to transfer results from the
well-developed theory of quasi-hereditary algebras to the new class of
Reedy algebras. That this is feasible is demonstrated by 
the second main result, Theorem~\ref{thm:main}. This characterises the existence of a Reedy
decomposition of an algebra $A$
recursively via Reedy
decompositions of $eAe$ and $A/AeA$, where $e$ is any idempotent
of $A$ generating an ideal in the defining heredity chain of $A$ (viewed as
quasi-hereditary algebra).
This is an analogue of a result of Dlab and Ringel \cite[Theorem 1]{DLcomp} who
gave a characterisation of when an algebra $A$ is quasi-hereditary with $AeA$
somewhere in its heredity chain, in terms of the quasi-heredity of $A/AeA$ and
$eAe$ together with some additional conditions.

\section{Quasi-hereditary algebras and particular subalgebras}
\label{sec:two}
Fix a finite-dimensional algebra $A$ over a field $k$. The Jacobson radical of $A$ is denoted by $\rad(A)$. Modules over $A$ are always assumed to be finitely generated left $A$--modules. 
Let $L\coloneqq\{L(i)\, |\, i\in I\}$ be a set of representatives $L(i)=L^A(i)$ of the isomorphism classes of simple $A$--modules and, for all $i\in I$, denote by $P(i)$ or $P^A(i)$ the projective cover of $L(i)$. Sometimes the elements of the index set $I$ and also the simples $L(i)$ for $i\in I$ will be called the \textit{weights} of $A$. For an $A$--module $M$ we set $\mathrm{top}(M)\coloneqq M/\rad(M)$.

Following \cites{CPS,DR}, a (two-sided) ideal $J$ of $A$, is called a \textit{heredity ideal} if the following
conditions are satisfied:
\begin{itemize}
\item[(i)] $J$ is an idempotent ideal,  i.e. $J^2=J$,
\item[(ii)] $J\rad(A)J=0$,
\item[(iii)] $J$ is projective as a left $A$--module.
\end{itemize}
The algebra $A$ is called \textit{quasi-hereditary} if there exists a \textit{heredity chain} of $A$, that is, a chain of (two-sided) ideals 
\begin{equation*}
0=J_{-1}\subseteq J_0\subseteq \cdots \subseteq J_n=A
\end{equation*}
such that $J_j/J_{j-1}$ is a heredity ideal in $A/J_{j-1}$ for all $j=0,1,\dots,n$. One may always assume that the heredity chain is of the form 
\begin{equation*}
0\subseteq A\varepsilon_0A \subseteq A(\varepsilon_0+\varepsilon_1)A \subseteq \cdots \subseteq A(\varepsilon_0+\varepsilon_1+\cdots+\varepsilon_n)A=A
\end{equation*}
where $\varepsilon_0,\varepsilon_1,\dots,\varepsilon_n$ form a complete set of pairwise orthogonal
idempotent elements of $A$, see for instance \cite[\S1]{japan}. It can be shown that there exists a unique $l(i)\in\{0,1,\dots,n\}$ for each simple module $L(i)$, such that $L(i)$ is a simple composition factor of $\mathrm{top}(J_{l(i)}/J_{l(i)-1})$. This is used to define a partial order $\unlhd$ on the set $L$ of simple modules by setting $L(i)\lhd L(j)$ if and only if $l(i)>l(j)$. Instead of $L(i)\lhd L(j)$, we sometimes write $i\lhd j$ or even $e_i\lhd e_j$ for $\{e_i\colon i\in I\}$ a complete set of primitive orthogonal idempotents satisfying $P(i)\cong A e_i$.
An algebra $A$ may have many quasi-hereditary structures. Fixing
  one means fixing the partial order $\unlhd$. Therefore, it is more precise
  to define a quasi-hereditary algebra as a pair
  $(A,\unlhd)$.

For all $i\in I$, the \textit{standard module} $\Delta(i)$ is defined to be $\Delta(i)\coloneqq P(i)/\sum_{j\not\unlhd i}\mathrm{Tr}_{P(j)}P(i)$, where $\mathrm{Tr}_{P(j)}P(i)$ is the trace of $P(j)$ in $P(i)$, that is, the sum of the images of all homomorphisms from $P(j)$ to $P(i)$. Thus, $\Delta(i)$ is the largest quotient of $P(i)$ having composition factors $L(j)$ with $j\unlhd i$. 
Sometimes we write $\Delta^A(i)$ to emphasise that the standard modules are defined over $A$. The standard modules over a quasi-hereditary algebra have the following fundamental properties which originate from \cite[Lemma~3.4]{CPS}, see also \cite{DRapp}. For all $i\in I$ there exist the following short exact sequences: 
\begin{itemize}
\item[(i)] $0\rightarrow K(i)\rightarrow \Delta(i) \rightarrow L(i)\rightarrow 0$ where $K(i)$ admits a finite filtration whose subquotients are isomorphic to simple modules of the form $L(j)$ for $j\lhd i$,
\item[(ii)] $0\rightarrow Q(i) \rightarrow P(i)\rightarrow \Delta(i)\rightarrow 0$ where $Q(i)$ admits a finite filtration whose subquotients are isomorphic to standard modules of the form $\Delta(j)$ for $j\rhd i$.
\end{itemize}

Quasi-hereditary algebras arise frequently in various areas of mathematics,
including algebraic Lie theory, where Verma modules of semisimple complex Lie
algebras and Weyl modules of reductive algebraic groups are standard modules
of certain quasi-hereditary algebras. In an attempt to establish a parallel
with the classical Poincar\'e--Birkhoff--Witt theorem for complex semisimple Lie algebras, an analogue of Borel subalgebras has been defined in the context of quasi-hereditary algebras (\cites{Koe1,Koe2}). 
Let $(A,\unlhd)$ be a quasi-hereditary 
structure on the finite-dimensional algebra $A$ with index set $I$ of
simple modules. A unital subalgebra $B \subseteq A$ with the same index set of
simples is called an \textit{exact Borel subalgebra} of $(A,\unlhd)$ if
induction $A \otimes_B -$ from $A$--modules to $B$--modules is an exact
functor sending simple $B$--modules $L^B(i)$ to standard $A$--modules
$\Delta^A(i)$ for each index $i$, and $B$ itself is a quasi-hereditary
algebra, for the given ordering ${\unlhd}$, with simple standard modules
$\Delta^B(i) = L^B(i)$.

The opposite algebra of a quasi-hereditary algebra is quasi-hereditary too  with the same heredity chain (hence for the same ordering on the simples). A
subalgebra $C$ of $A$ is called a \textit{Delta subalgebra} if its
opposite algebra is an exact Borel subalgebra of $A^{op}$. The algebra $C$
is then quasi-hereditary, for the given ordering ${\unlhd}$, with projective
standard modules. The latter coincide with the standard modules over $A$,
that is, the restriction of $\Delta^A(i)$ to $C$--modules is the
indecomposable projective $C$--module $P^C(i)$.

Not every quasi-hereditary algebra has an exact Borel subalgebra or a
Delta subalgebra; an example that first appeared in \cite{Koe1} is
given below. In \cite{KKO} it has, however, been shown that for
every quasi-hereditary algebra $A$ there is a Morita equivalent one, with
a quasi-hereditary structure provided by transport of structure from $A$,
that does have an exact Borel subalgebra, carrying additional
structure (see also \cites{Conde,KueMie,RODRIGUEZRASMUSSEN} for further information).

\section{Reedy algebras}

The following class of algebras was introduced in \cite{DS}, replacing
the unique down-up factorisation property of Reedy categories by a linearised
version suitable for algebras over a field.

\begin{definition}
\label{def:Reedy}
Let $A$ be a finite-dimensional $k$--algebra with a complete set $E\coloneqq\{e_0,e_1,\dots,e_n\}$ of pairwise orthogonal idempotents. We call $A$ \textit{Reedy} (or a
\textit{Reedy algebra}) if it admits a degree function ${\deg\colon E}\rightarrow \mathbb{N}$ and two subalgebras $A^+$ and $A^-$ containing the idempotents $e_0, \dots, e_n$ 
such that the following conditions are satisfied:
\begin{itemize}
\item[(i)] for all $i$, there is an isomorphism of $k$--vector spaces $e_iA^+e_i\cong k$ and for $i\neq j$ the implication $e_jA^+e_i\neq 0\Rightarrow \deg(e_j)>\deg(e_i)$ holds true,
\item[(ii)] for all $i$, there is an isomorphism of $k$--vector spaces $e_iA^-e_i\cong k$ and for $i\neq j$ the implication $e_jA^-e_i\neq 0\Rightarrow \deg(e_j)<\deg(e_i)$ holds true,
\item[(iii)] for each pair $i,j$, the multiplication in $A$ induces an
  isomorphism of $k$--vector spaces,
\begin{equation}
\label{eq:Reedy_decomp}
\bigoplus_{l=0}^{n} \,e_jA^+e_l\, \otimes_{k}\, e_lA^-e_i \rightarrow e_jAe_i.
\end{equation}

\end{itemize}
Sometimes the Reedy algebra $A$ will be denoted by $(A,A^+,A^-)$ and the
collection of isomorphisms in \eqref{eq:Reedy_decomp} will be called a 
\textit{Reedy decomposition}.
\end{definition}

Note that, in the definition above, the idempotents $e_i \in A$ are not necessarily primitive. Moreover, it follows from the definition that $A^+$ and $A^-$ have the same unit as $A$. To illustrate the definition of Reedy algebras, some positive or negative examples are now given.

\begin{example}
\label{ex:matrix}
This example shows that the property of being a Reedy algebra is not invariant under Morita equivalence. 

Let $k$ be a field and consider the $k$-algebra $A=k$.
  Choosing $e=1$ and
  $A^+ = k = A^-$ yields a Reedy decomposition of $A$.
  
  The $k$-algebra $B$ of two by two matrices over $k$ is Morita equivalent to
  $A$. It does, however, not admit a Reedy decomposition. Indeed, when choosing 
  the set $E$ to contain solely the
  idempotent $e= 1$, the algebras $B^+$ and $B^-$ must both
  coincide with $k$, by (i) and (ii), respectively, and condition (iii)
  is not satisfied, as $B$ cannot be isomorphic to $k \otimes_k k$.
  
  Otherwise,
  the unit $1 \in B$ must be decomposed as a sum $1=e+f$ of two pairwise
  orthogonal idempotents, which are necessarily primitive in $B$. Both $e$ and $f$ must be contained in $B^+$ and in $B^-$. That forces the semisimple algebra
  $S \coloneqq k e \oplus k  f$ to be a two-dimensional subalgebra of both $B^+$ and $B^-$. Hence, the dimensions of $B^+$ and of $B^-$ must be at least two each.

  If $B^+$ is two-dimensional, then it equals $S$ and
  the assumed Reedy decomposition of $B$ can be written as $S \otimes_S B^-$, which is isomorphic 
  to $B^-$. This implies $B^- = B$, but $B$ itself cannot satisfy condition (ii), whatever the degree on $e$ and $f$ is. 
  If $B^-$ is two-dimensional, a similar contradiction arises.

  As we have already seen, a contradiction to condition (ii) or (i) arises when $B^-$ or $B^+$ have dimension four. So, the only case being left is that of $B^+$ and $B^-$ both having dimension three. Because of conditions (ii) and (iii) using opposite orderings, the algebras $B^+$ and $B^-$ cannot be equal. In particular, their intersection must be precisely $S$.
  Now condition (iii) of the assumed Reedy decomposition leads to a contradiction in the following way. The three vector spaces $eBe$, $eB^+e$ and $eB^-e$ are one-dimensional each and so are the vector spaces $fBf$, $fB^+f$ and $fB^-f$. By definition of $B$ and by $e$ and $f$ being primitive, also $eBf$ and $fBe$ are one-dimensional. As $B^+$ is three-dimensional, either $eB^+f$ or $fB^+e$ is one-dimensional and the other one vanishes, and similarly for $B^-$. Since $B^+ \cap B^- = S$ and $fBe$ is one-dimensional, either $fB^+e$ or $fB^-e$ must vanish. Up to possibly exchanging $e$ and $f$, this implies that $fB^+e$ and $eB^-f$ both are one-dimensional, while $eB^+f$ and $fB^-e$ vanish. However, condition (iii) requires an isomorphism of $k$-vector spaces $(fB^+e \otimes_k eB^-f) \oplus  (fB^+f \otimes_k fB^-f) \cong fBf$ between a two-dimensional vector space and a one-dimensional one, which is a contradiction.
  
\end{example}

\begin{example}
\label{ex:Steffens_example}
This example illustrates how the existence and shape of a Reedy
   decomposition depends on the chosen ordering of the given
   idempotents. \medskip

\hspace{-10pt}\begin{minipage}{0.55 \textwidth}
\setlength{\parindent}{10pt}
  Let the algebra $A\coloneqq kQ/I$ be given by the quiver $Q$ on the right
    and a commutativity relation between the two paths of
  length two. By $S$ we denote the commutative semisimple $k$-algebra of
  dimension four that is generated by the four vertices. \\
  \end{minipage}
  \begin{minipage}{0.45 \textwidth}
  \hspace{2cm}
  \vspace{0.4cm}
  \xymatrix@C=1pc@R=1pc{
    & a \ar[ld]\ar[rd] & \\
    b \ar[rd] && c\ar[ld]  \\
    & d & } 
    \end{minipage} 
  
  For the degree function $\deg(e_a)=1$, $\deg(e_b)=2$, $\deg(e_c)=3$ and $\deg(e_d)=4$, the algebra $A$ is Reedy with $A^+ = A$ and $A^- = S$. Choosing instead $\deg(e_a)=1$, $\deg(e_b)=2$, $\deg(e_c)=2$ and $\deg(e_d)=3$, the algebra $A$ is Reedy again with $A^+ = A$ and $A^- = S$. This shows that degree functions do not have to be injective.
 
   For the degree function $\deg(e_a)=4$, $\deg(e_b)=2$, $\deg(e_c)=3$ and $\deg(e_d)=1$, the algebra $A$ is Reedy with $A^+ = S$ and $A^- = A$. 
  
 For the degree function $\deg(e_a)=4$, $\deg(e_b)=3$, $\deg(e_c)=1$ and $\deg(e_d)=2$, the algebra $A$ is not Reedy for any choice of $A^+$ and $A^-$. Indeed, suppose that such subalgebras existed. Using degree considerations and taking into account the shape of the quiver $Q$, condition (iii) of Definition \ref{def:Reedy} would yield $e_bA^-e_a=e_bAe_a$. In a similar manner, one would obtain $e_dA^-e_b=e_dAe_b$ and $e_cAe_a=e_cA^-e_a$. Consequently, $ k\cong e_dAe_bAe_a=e_dA^-e_bA^-e_a$ would be a subspace of $e_dA^-e_a$. But then, by condition (iii), the vector space $e_dAe_a$ would be at least two-dimensional, which is a contradiction. An alternative argument is as follows. Consider the ordering on the simple $A$-modules that reflects the chosen degree function. In \cite[Example 2.3]{Koe1}, it is shown that this ordering gives $A$ a
   quasi-hereditary structure for which there is no exact
   Borel subalgebra (the notation used in \cite{Koe1} corresponds to the one used here, as explained later on in  Corollary~\ref{cor:heredity chain of Reedy}).  This, together with \cite[Theorem 4.22]{DS}, also shows that the algebra $A$ is not Reedy for this degree function. 
\end{example}

\begin{example}
This example sketches a construction that can be
   used to produce Reedy algebras $A$ from given algebras $A^+$ and $A^-$.
   
Let $A^+$ and $A^-$ be basic algebras contained in some other algebra $\Lambda$. Assume that $\Lambda$ contains a set $E = \{e_0, \dots , e_n\}$ of primitive idempotents, which are also contained in
$A^+$ and $A^-$. Suppose there is a degree function on $\Lambda$ such that $A^+$ and $A^-$ satisfy
conditions (i) and (ii) in the definition of Reedy algebra, with respect to $E$. Set $S \coloneqq k e_0 \times \dots\times k e_n$ and assume that $S = A^+ \cap A^-$ (where
the intersection is taken in $\Lambda$) is a maximal
semisimple subalgebra of $A^+$ as well as of $A^-$. 
This yields decompositions $A^+ = S \oplus \rad(A^+)$ and $A^- = S \oplus \rad(A^-)$.

Forgetting now about $\Lambda$, one may set
$A\coloneqq\bigoplus_{l=0}^n A^+ e_l \otimes_k e_l A^-$
as a vector space, and turn $A$ into an associative $k$-algebra in the
following way. Embed $A^+$ into $A$ by sending an element $b \in A^+$
to $\sum_{l=0}^nb e_l \otimes_k e_l$, and similarly for $A^-$: $ c \mapsto \sum_{l=0}^n e_l \otimes_k e_lc$. 
Now define a product $\times$ on $A$ by setting the product 
$b \times c \coloneqq 0$ when $b \in \rad(A^-)$ and $c \in \rad(A^+)$ and using the product in $S$ for $c, b \in S$. 
In terms of elements, the product looks as follows. Consider $\alpha,c+ u\in A^+$ and $b+t,\delta \in A^-$, where
$t,u \in S$, $c \in \rad{A^+}$ and $b\in \rad{A^-}$. Then the product
$(\alpha e_l \otimes e_l (b+t)) \times ((c+u) e_m \otimes e_m \delta)$ has the following summands. The summand
$(\alpha e_l \otimes e_l b) \times (c e_m \otimes e_m \delta) $ is zero, since $b$ and 
$c$ are elements of the respective radicals. 
The summand $(\alpha e_l \otimes e_l b) \times (u e_m \otimes e_m \delta)$ equals
$\alpha e_l \otimes e_l b u e_m \delta$, where
$e_l b u e_m \delta$ is a product in $A^-$. Finally, the summand $(\alpha e_l \otimes e_l t) \times ((c+u) e_m \otimes e_m \delta)$ equals
$\alpha e_l t (c+u) e_m \otimes e_m \delta$, where $\alpha e_l t (c+u)e_m$ is a
product in $A^+$.
The embeddings defined above allow to view $A^+$ and $A^-$ as subalgebras of $A$. 
   Then $A$
   is a Reedy algebra. (When $A^-$ is the opposite algebra of $A^+$, the
   algebra $A$ is a dual extension algebra in the sense of Xi, see
   \cites{Xidual,DengXidual}.)
\end{example}
   
\begin{example}
The construction sketched in the previous example can
sometimes be generalised, using assumptions that have 
been worked out in special cases.

   For algebras $A^+$ and $A^-$ as in the previous example,
   satisfying conditions (i) and (ii) in
   the definition of Reedy algebra, one may again set
   $A\coloneqq\bigoplus_{l=0}^n A^+ e_l \otimes_k e_l A^-$,
   as a vector space,  containing $A^+$ and $A^-$ (as subspaces).
   But now one may try to define the product on $A$ more generally, by setting
   $b \times c$ for $b \in \rad(A^-)$ and $c \in \rad(A^+)$ to be a possibly non-trivial linear combination of elements in $A$, and again keeping the multiplications in $A^+$, $A^-$ and $S$. To 
   get in this way an associative algebra $A$ needs strong assumptions on the linear combinations being used in the definition of this multiplication. If $A$ becomes an
    associative algebra in this way, it contains $A^+$ and $A^-$ as subalgebras and automatically
    has the desired Reedy decomposition.
    Such a construction has been carried out, under strong assumptions,
    for instance in \cite{KXtwisted}, and shown to work for Temperley--Lieb algebras as well as for q-Schur algebras of finite representation type. Other examples of algebras with such a Reedy decomposition are the twisted double incidence algebras introduced and studied by Deng and Xi (\cite{DengXi}).
\end{example}

A finite-dimensional $k$--algebra $\Lambda$ is called \textit{elementary} if the $k$--algebra $\Lambda/\rad(\Lambda)$ is isomorphic to a product of copies of the field $k$ (or, equivalently if every simple $\Lambda$--module is one-dimensional over $k$). This implies that the algebra $\Lambda$ is basic, and the converse holds under the additional assumption that $k$ is an algebraically closed field.

\begin{lemma}
\label{lem:basic_algebras}
Let $(A,A^+,A^-)$ be a Reedy $k$--algebra with a complete set $E\coloneqq\{e_0,e_1,\dots,e_n\}$ of pairwise orthogonal idempotents, as in Definition~\ref{def:Reedy}. Then the algebras $A^+$ and $A^-$ are elementary, both having $E$ as a complete set of primitive pairwise orthogonal idempotents.

As a consequence, the $k$-algebra $S\coloneqq ke_0 \times ke_1 \times \dots \times
ke_n$ is a maximal semisimple subalgebra of $A^+$ and of $A^-$ (but in general not of $A$).
\end{lemma}

\begin{proof}
By Definition \ref{def:Reedy}(i), $\End_{A^+}(A^+e_i) \cong k$ and the existence of an invertible homomorphism in $\Hom_{A^+}(A^+e_i,A^+e_j)\cong e_iA^+e_j$ for $i\neq j$ would imply that $\deg(e_j)>\deg(e_i)$ and $\deg(e_i)>\deg(e_j)$. Hence the modules $A^+e_i$ and $A^+e_j$ are indecomposable and can only be isomorphic if $i=j$. This proves the statement for $A^+$. The corresponding statement for $A^-$ can be proved by similar arguments, using Definition \ref{def:Reedy}(ii).
\end{proof}

\begin{remark}
\label{rem:multiplications}
If $(A,A^+,A^-)$ is a Reedy algebra  with a complete set $E\coloneqq\{e_0,e_1,\dots,e_n\}$ of pairwise orthogonal idempotents, then multiplication in $A$ induces an isomorphism of left $A^+$--modules, 
${\bigoplus_{l=0}^{n} \,A^+e_l\, \otimes_{k}\, e_lA^-e_i} \rightarrow Ae_i$, and an isomorphism 
$\bigoplus_{l=0}^{n} \,e_iA^+e_l\, \otimes_{k}\, e_lA^- \rightarrow e_iA$ of right $A^-$--modules. 
Multiplication also yields an isomorphism ${\bigoplus_{l=0}^{n} \,A^+e_l\, \otimes_{k}\, e_lA^-} \rightarrow A$ of $A^+$--$A^-$--bimodules. In addition, if $e$ is an idempotent of minimal degree in $E$, then for each $e'\in E$, by the defining properties of Reedy algebras, there are isomorphisms $eAe'\cong eA^- e'$ and $e'Ae\cong e'A^+e$ as $k$--vector spaces (which can be taken to be equalities).
\end{remark}

The following observation is well known.

\begin{proposition}
\label{prop:directed}Let $A$ be a finite-dimensional $k$--algebra with an irredundant list $P(0),P(1),\dots,P(n)$ of representatives of isomorphism classes of indecomposable projective modules. Assume that there is a partial order $\unlhd$ on the set of weights of $A$ with $\End_{A}(P(i))\cong k$ for all $i=0,1,\dots,n$ and such that 
$\Hom_{A}(P(j),P(i))\neq 0$ only if $i \unlhd j$. Then $(A,\unlhd)$ is quasi-hereditary with simple standard modules. Moreover,
$(A,\unlhd^{op})$ is also quasi-hereditary with projective standard modules (when using
the opposite partial order of weights).
  
In particular, if $(A,A^+,A^-)$ is Reedy then $A^+$ (resp. $A^-$) is quasi-hereditary with projective (resp. simple) standard modules, by setting $i \lhd j$ if and only if $\deg(e_i)>\deg(e_j)$.
\end{proposition}

\begin{proof}
The proofs of the first two statements proceed by induction on $n$. 
  The first
  statement is shown as follows. Let $t$ be a maximal weight with respect to $\unlhd$. Since $\Hom_{A}(P(j),P(i))\neq 0$ only if $i \unlhd j$, all composition factors of $P(t)$ are isomorphic to $L(t)$, and by $\End_{A}(P(t))\cong k$, the composition
  multiplicity of $L(t)$ in $P(t)$ must be one, which implies
  $L(t) = P(t)$. The trace $J$ of the simple
  projective module $P(t)$ in $A$ is a sum, hence a direct sum, of copies
  of $P(t)=L(t)$ and thus projective as a left module. Moreover it is generated,
  as a two-sided ideal,
  by an idempotent generating $P(t)$ and it has vanishing radical. Therefore,
  $J$ is a heredity ideal and the module $P(t) = L(t)$ is the corresponding
  standard module. The quotient $A/J$ satisfies an analogous condition and
  has one indecomposable projective module less. Hence, $(A,\unlhd)$
  is quasi-hereditary with simple standard modules.

  For the second statement, consider a minimal weight $b$ with respect to $\unlhd$. Since by assumption $\Hom_{A}(P(j),P(i))\neq 0$ only if $i \unlhd j$, the trace $J'$
  of $P(b)$ in $A$ is a direct sum of copies of $P(b)$, as $P(b)$ does not
  map non-trivially
  to any other projective module. Thus, $J'$ is projective as a left
  $A$--module. By the assumption $\End_{A}(P(b))\cong k$, the radical of $P(b)$ does
  not have composition factors of type $L(b)$, which implies that
  $J' \rad(A) J' = 0$. As $J'$ is generated, as a two-sided ideal,
  by an idempotent generating $P(b)$, it is a heredity ideal.
  The quotient $A/J'$ satisfies an analogous condition and has one
  indecomposable projective module less. Therefore, it is quasi-hereditary
with projective standard modules, when using
the opposite partial order on weights.

The third statement is a special case of the first and the second one.
\end{proof}

\section{Characterising Reedy algebras in terms of quasi-hereditary structures}

By \cite[Theorem 4.22]{DS} Reedy algebras are quasi-hereditary.
The converse is not true, as illustrated by Examples \ref{ex:matrix} and
\ref{ex:Steffens_example}
above. This raises the question of how to characterise Reedy algebras as
quasi-hereditary algebras with additional structure.
Theorem~\ref{thm:characterization} answers this question, characterises the
subalgebras $A^+$ and $A^-$ of a Reedy algebra by strong properties
and also gives a
different proof of quasi-heredity of Reedy algebras via \cite{Koe2}. In addition, Corollary \ref{cor:heredity chain of Reedy} provides yet another proof of quasi-heredity of Reedy algebras. The
proof in \cite{DS} is based on an analysis of standard modules while the
proofs in this section focus on heredity chains.

\begin{theorem}
\label{thm:characterization}
Let $A$ be a finite-dimensional algebra over a field $k$ and let $B$ and $C$ be unital $k$--subalgebras of $A$. Then the following are equivalent:
\begin{itemize}
\item[(i)] The algebra $A$ is Reedy with $A^-=B$ and $A^+=C$. 
\item[(ii)] The subalgebras $B$ and $C$ are elementary, $S\coloneqq B\cap C$ is a maximal semisimple subalgebra of $B$ and $C$ and the multiplication in $A$ induces an isomorphism
\begin{equation}
C\otimes_{S}B\rightarrow A,
\nonumber
\end{equation}
of $C$--$B$--bimodules. In addition, after identifying the weights of $B$ with those of $C$,  the algebra $B$ is quasi-hereditary with simple standard modules and $C$ is quasi-hereditary with projective standard modules for the same partial order. 
\item[(iii)] The subalgebras $B$ and $C$ are elementary, $S\coloneqq B\cap C$ is a maximal semisimple subalgebra of both $B$ and $C$, and there are bijections between the weights of $A$, $B$ and $C$. In addition, the algebra $A$ is quasi-hereditary, $B$ is an exact Borel subalgebra of $A$ and $C$ is a Delta subalgebra of $A$, where the partial orders on the weights of $B$ and $C$ are the same. 
\end{itemize}
\end{theorem}

\begin{proof}
(i)$\Rightarrow$(ii)  For any Reedy algebra $A$, by Lemma \ref{lem:basic_algebras} the algebras $B\coloneqq A^-$ and $C\coloneqq A^+$ are elementary and $S$ is a maximal semisimple subalgebra of both $B$ and $C$. The claimed quasi-heredity of $B$ and $C$ follows from Proposition~\ref{prop:directed}. 
The Reedy decomposition of $A$ can be rewritten as a collection of isomorphisms $e_jC\otimes_{S}Be_i\cong e_jAe_i$ of $k$--vector spaces given by multiplication; the desired isomorphism $C\otimes_{S}B\rightarrow A$  now follows from Remark~\ref{rem:multiplications}.

(ii)$\Rightarrow$(i) The restriction to $S$ of the canonical epimorphism $C\to C/\rad C$ has kernel $S \cap \rad C$ which is a nilpotent ideal of $S$ and hence contained in $\rad S=0$. As $C$ is elementary, $C/\rad C$ is isomorphic to either $S$ or to $S \times T$ where $T$ is another elementary semisimple subalgebra. Lifting idempotents in $T$ to idempotents in $C$ would contradict the assumption of $S$ being a maximal semisimple subalgebra of $C$. Thus $C/\rad C\cong S$ and, in the same way, $B/\rad B \cong S$. Like $C/\rad C$, the semisimple algebra $S$ is elementary and thus can be written as $S= ke_0\times ke_1 \cdots \times ke_n$. Here $E:=\{e_0,e_1,\dots,e_n\}$ is a complete set of primitive pairwise orthogonal idempotents of $S$, $B$ and $C$ and a complete set of pairwise orthogonal idempotents of $A$.

We may assume that $Ce_0C\subseteq C(e_0+e_1)C \subseteq \cdots \subseteq C(e_0+e_1+\cdots+e_n)C=C$ is a heredity chain  of $C$ where $e_0\rhd e_1 \rhd\cdots \rhd e_n$ in the quasi-hereditary order of $C$; cf.~\cite[Proposition~1.3]{japan}. Since by assumption the standard modules of $C$ are the projectives $Ce_i$ for all $i=0,1,...,n$, it follows from \cite[Lemma~1.6]{DRapp} and from $C$ being elementary that $e_iCe_i\cong k$ as $k$--vector spaces. Furthermore, by \cite[Lemma~1.2]{DRapp}, the non-vanishing of $\Hom_{C}(Ce_j,\Delta^C(i))=e_jCe_i$ for $i\neq j$ implies $e_j\lhd e_i$. Moreover, $B$ is quasi-hereditary with respect to the same order as that of $C$ and has simple standard modules. Thus, a heredity chain of $B$ is given by $Be_0B\subseteq B(e_{0}+e_{1})B \subseteq \cdots \subseteq B(e_0+e_{1}+\cdots+e_n)B=B$, where for all $i$ there is an isomorphism of $k$--vector spaces $e_iBe_i\cong k$ and, for $i\neq j$, non-vanishing of $\Hom_{B}(Be_j, Be_i)\cong e_jBe_i$ implies $e_i\lhd e_j$. We define a degree function $\deg\colon E\rightarrow\mathbb{N}$ by $\deg(e_i):=i$. The subalgebras $A^+\coloneqq C$ and $A^-\coloneqq B$ of $A$ satisfy conditions (i) and (ii), respectively, in Definition \ref{def:Reedy}. Finally, for all $i,j$, the isomorphism $C\otimes_{S}B\rightarrow A$, induces an isomorphism 
$e_jC \otimes_{S}Be_i\rightarrow e_jAe_i$ of $k$--vector spaces given by multiplication, which is the meaning of condition (iii) in Definition \ref{def:Reedy}.

(ii)$\Rightarrow$(iii) We may use the already proved implication (ii)$\Rightarrow$(i) which tells us that $A$ is Reedy with $A^-=B$, $A^+=C$ and a complete set of pairwise orthogonal idempotents $E=\{e_0,e_1,\dots,e_n\}$. By \cite[Lemma~4.12]{DS} for all $i=0,1,\dots,n$ the $A$--module $Ae_i/\sum_{\deg(e_l)<\deg(e_i)}\mathrm{Tr}_{Ae_l}Ae_i$ has a unique simple factor $L_{e_i}$ and by \cite[Theorem 4.14]{DS} the assignment $e_i\mapsto L_{e_i}$ defines a bijection between $E$ and the isomorphism classes of  simple left $A$--modules. The cited statements are proved in the language of functor categories but they can be applied to the category of left $A$--modules; for further details see for instance the paragraph before \cite[Lemma~4.21]{DS}. By combining this with Lemma \ref{lem:basic_algebras} we deduce that there are bijections between the weights of $A$, $B$ and $C$. Thus we can now apply \cite[Theorem~4.1]{Koe2}, which is valid without the additional assumption that $k$ is algebraically closed  as we are assuming that $B$ and $C$ are elementary.

(iii)$\Rightarrow$(ii) One can apply directly \cite[Theorem~4.1]{Koe2} (again, the assumption that $k$ is algebraically closed is redundant). 
\end{proof}

\begin{remark}
In the proof of implication (ii)$\Rightarrow$(iii) in Theorem~\ref{thm:characterization} we refer to \cite[Theorem~4.1]{Koe2}. In this reference, the Cartan decomposition of $A$ is used to show that the ideal $J_t$ generated by the idempotents of highest weight is heredity, and that $A/J_t$ admits a Cartan decomposition too (then one proceeds by induction). This is akin to the strategy used in Corollary \ref{cor:heredity chain of Reedy} appearing later in this section, which provides yet another proof of quasi-heredity of Reedy algebras, making the corresponding heredity chains more transparent.
\end{remark}

\begin{remark}
  Theorem \ref{thm:characterization}
  implies that the concept of a Reedy decomposition of a finite-dimensional $k$-algebra coincides with the  concept of a Cartan
  decomposition of a quasi-hereditary algebra used in \cite[Theorem~4.1]{Koe2}. The term Reedy decomposition
  appears to fit much better to this kind of triangular decomposition, and
  thus we would like to suggest to use only this term in the future.
\end{remark}

Theorem~\ref{thm:characterization} provides many new examples of
quasi-hereditary algebras with an exact Borel subalgebra and a Delta
subalgebra. Conversely, it also provides new examples of Reedy algebras.

\begin{example}
  Monomial algebras form a frequently studied class of 
  finite-dimensional algebras, given by quivers and particular relations. It is proved in \cite[Theorem 3.7]{Rasmussen} that every quasi-hereditary monomial algebra has a Reedy decomposition. The main ingredient of the proof is an explicit construction of an exact Borel subalgebra in terms of certain paths in the quiver, and a similar construction of a Delta subalgebra. These subalgebras then produce a Reedy decomposition, which in this situation can be checked directly, in the spirit of Theorem \ref{thm:characterization}.
\end{example}

Another way of producing new Reedy algebras is by taking tensor products.

\begin{proposition}
\label{prop:tensor}
Tensor products of Reedy algebras are Reedy algebras.
\end{proposition}

\begin{proof}
  Let $(A,A^+,A^-)$ and $(B,B^+, B^-)$ be Reedy algebras over the field $k$
  with respective sets of idempotents $\{ e_0, \dots, e_n\}$ and
  $\{ f_0, \dots, f_m\}$. We show that $C\coloneqq A \otimes_k B$ is Reedy with
  respect to the following data: $C^+\coloneqq A^+ \otimes_k B^+$ and
  $C^-\coloneqq A^- \otimes_k B^-$ are subalgebras of $C$ with common set of pairwise
  orthogonal idempotents $e_i \otimes_k f_j$, $0 \leq i \leq n$, $0 \leq j \leq m$,
  and degree function $\deg (e_i \otimes_k f_j) \coloneqq \deg(e_i) + \deg(f_j)$.

  Then $(e_i \otimes_k f_j) C^+ (e_i \otimes_k f_j) =
  (e_i \otimes_k f_j) (A^+ \otimes_k B^+) (e_i \otimes_k f_j) =
  e_i A^+ e_i \otimes_k f_j B^+ f_j \cong k \otimes_k k \cong k$, and
  similarly for $C^-$.

  Moreover, non-vanishing of $(e_i \otimes_k f_j) C^+ (e_h \otimes_k f_l) =
  e_i A^+ e_h \otimes_k f_j B^+f_l$ implies the inequality
  ${\deg(e_i) + \deg(f_j) \geq
  \deg(e_h) + \deg(f_l)}$ and equality only holds true for $e_i = e_h$ and
  $f_j = f_l$, and similarly for $C^-$.

  Finally, the Reedy decompositions of $A$ and $B$ imply a Reedy decomposition
  of $A \otimes_k B$, since componentwise multiplication induces
  for all $j,l,h,s$ isomorphisms:
  \begin{align*}
  &\bigoplus_{i,t} (e_j \otimes_k f_l) C^+ (e_i \otimes_k f_t) \otimes_k
  (e_i \otimes_k f_t) C^- (e_h \otimes_k f_s)\\ = &\bigoplus_{i,t} (e_j A^+ e_i \otimes_k f_l B^+ f_t)
  \otimes_k (e_i A^-e_h \otimes_k f_t B^- f_s) \\
  \cong & e_j A e_h \otimes_k f_l B f_s \\
  = & (e_j \otimes_k f_l) (A \otimes_k B) (e_h \otimes_k f_s). 
  \end{align*}
\end{proof}

The rest of this section contains properties of Reedy algebras to be used in the sequel, motivated by similar properties of quasi-hereditary algebras.

\begin{lemma}
\label{lem:plus_minus}
For a Reedy algebra $(A,A^+,A^-)$ with a degree function $\deg\colon E=\{e_0,e_1,\dots,e_n\}\rightarrow\mathbb{N}$, let $t\coloneqq\min\{\deg(e)\colon e \in E\}$ and 
let $\varepsilon_t$ be the sum of the elements of $E$ having degree precisely $t$. Then $J_t\coloneqq A\varepsilon_tA$ is a heredity ideal and 
 the multiplication map $$A^+\varepsilon_t \otimes_{S_t} \varepsilon_t A^-\rightarrow A\varepsilon_t A$$ is an isomorphism of $A^+$--$A^-$--bimodules.
\end{lemma}

\begin{proof}
Let $e, e'$ be elements in $E$ of minimum degree $t$. By the defining property (iii) of Reedy algebras there is an isomorphism of $k$--vector spaces, 
 \begin{equation}
e'Ae \cong \bigoplus_{i=0}^{n} \,e'A^+e_i\, \otimes_{k}\, e_iA^-e. 
\nonumber
\end{equation}
Since the degrees of $e$ and $e'$ are minimal, by the defining properties (i) and (ii) of Reedy algebras, we obtain that $e'Ae$ is either zero or $e=e'$ in which case $eAe\cong k$. This proves that $\varepsilon_tA\varepsilon_t$ is isomorphic to the semisimple algebra $S_t\coloneqq \prod_{\deg(e_i)=t}k e_t$. In particular, $J_t \mathrm{rad}(A) J_t=0$. To prove that $J_t$ is a heredity ideal of $A$, it remains to show that it is a projective left $A$--module. By \cite[Statement~7]{DR}, it suffices to verify that the multiplication homomorphism $m:A\varepsilon_t\otimes_{S_t}\varepsilon_tA\rightarrow A\varepsilon_t A$ is injective or, equivalently, to show that the composition of $m$ with the inclusion $\iota$ of $A\varepsilon_t A$ in $A$ is injective. Notice that, by using Remark \ref{rem:multiplications} and parts (i) and (ii) of Definition \ref{def:Reedy},  for any idempotent $e\in E$ of degree $t$, we have $Ae \otimes_{k} eA= A^+e \otimes_{k} eA^-$, hence $A^+\varepsilon_t\otimes_{S_t}\varepsilon_tA^-= A\varepsilon_t\otimes_{S_t} \varepsilon_tA$ as left $A^+$--$A^-$--bimodules. Thus, $\iota \circ m$ coincides with the restriction of the multiplication map $\bigoplus_{l=0}^n \,A^+e_l\, \otimes_{k}\, e_lA^- \rightarrow A$ to $\bigoplus_{\deg(e_l)=t} \,A^+e_l\, \otimes_{k}\, e_lA^-$, which is injective by Remark \ref{rem:multiplications}.
\end{proof}

Quasi-hereditary algebras are stable under certain natural constructions: the idempotent quotients and centraliser subalgebras arising from the inductive definition in terms of heredity ideals are again quasi-hereditary (see \cite[page 92 and Corollary 3.7]{CPS}). Consequently, each quasi-hereditary algebra comes with two chains of related
quasi-hereditary algebras, one formed by quotient algebras, the other one
by centraliser subalgebras.  Reedy algebras are now shown to come with
similar chains of related Reedy algebras.

 For this, we need to introduce first some notation. Suppose that $(A,A^+,A^-)$ is a Reedy algebra with a degree function $\deg\colon E=\{e_0,e_1,e_2,\dots,e_n\}\rightarrow\mathbb{N}$. We will denote by 
 $\varepsilon_l$ the sum of all elements of $E$ having degree $l$ (here $\varepsilon_l\coloneqq 0$ if no idempotent in $E$ has degree $l$). Sometimes we write $J_l\coloneqq A(\varepsilon_0 +\varepsilon_1+\cdots +\varepsilon_l )A$ and we also set  $J_{-1}\coloneqq 0$. The semisimple subalgebra of $S=A^+\cap A^-$ that is generated by the idempotents of degree $l$ is denoted by $S_l\coloneqq \prod_{\deg(e_i)=l}k e_i$. 

\begin{proposition}
\label{prop:more}
Suppose that $(A,A^+,A^-)$ is a Reedy algebra and let $e=\varepsilon_0+\varepsilon_1+\cdots + \varepsilon_l$. Then the following assertions hold true:
\begin{itemize}
\item[(i)] The algebra $(eAe,eA^+e,eA^-e)$ is a Reedy algebra with degree
  function inherited from~$A$. 
  
\item[(ii)] The algebra $(A/AeA,A^+/A^+eA^+,A^-/A^-eA^-)$ is a Reedy algebra
  with degree function inherited from $A$. 
  
\end{itemize}
\end{proposition}
\begin{proof}
  (i) The degree function of $eAe$ is the restriction of the degree function of
  $A$ to the set $E'$ consisting of all idempotents in $E=\{e_0,e_1,\dots,e_n\}$ having degree
  at most $l$. Then $eA^+e$ and $eA^-e$ satisfy properties (i) and (ii),
  respectively, in Definition \ref{def:Reedy}. Similarly, since $(A,A^+,A^-)$
  is Reedy, we have
 \begin{equation}
eAe \cong \bigoplus_{i=0}^{n} \,eA^+e_i\, \otimes_{k}\, e_iA^-e 
=  \bigoplus_{\deg(e_i)\leqslant l} \,eA^+e_i\, \otimes_{k}\, e_iA^-e.
\nonumber
\end{equation}
(ii)  Let now $E'\coloneq \{e_i +AeA\colon \deg(e_i)>l\}$ and define a degree function $\deg'$ on $E'$ by $\deg'(e_i+AeA)\coloneqq\deg(e_i)$.
For these data, the algebras $A^+/A^+eA^+$ and $A^-/A^-eA^-$ satisfy properties (i) and (ii),
respectively, from Definition \ref{def:Reedy}. Now, recall that
$S\coloneqq A^+\cap A^-$ and observe that
both $A^+/A^+eA^+$ and $A^-/A^-eA^-$ contain a maximal semisimple subalgebra isomorphic to $S/SeS$. We claim that the multiplication map 
\begin{equation*}
A^+/A^+eA^+ \otimes_{S/SeS}  A^-/A^-eA^- \rightarrow A/AeA
\end{equation*}
is an isomorphism of $k$-vector spaces. The proof proceeds by induction on $l$. By Lemma~\ref{lem:plus_minus}, the multiplication map $A^+\varepsilon_0\otimes_{S_0}\varepsilon_0A^-\rightarrow A\varepsilon_0A$ is an isomorphism. Note that $A^+\varepsilon_0\subseteq A^+\varepsilon_0A^+$ and that $A^+\varepsilon_0A^+$ decomposes as $\bigoplus_{i=0}^n A^+\varepsilon_0A^+ e_i$. Comparing degrees implies that  $A^+\varepsilon_0A^+=\bigoplus_{\deg(e_i)=0} A^+\varepsilon_0A^+ e_i\subseteq A^+\varepsilon_0$, so $A^+\varepsilon_0=A^+\varepsilon_0A^+$. Similarly, we have
$\varepsilon_0A^-=A^-\varepsilon_0A^- $, and thus also $A^+\varepsilon_0\otimes_{S_0}\varepsilon_0A^-=A^+\varepsilon_0A^+\otimes_{S_0}A^-\varepsilon_0A^-$. Hence the isomorphism $A^+\otimes_SA^-\rightarrow A$ restricts to an isomorphism
$\alpha: A^+\varepsilon_0A^+\otimes_{S_0}A^-\varepsilon_0A^-\cong A\varepsilon_0A$. Therefore, there is an induced isomorphism $\beta: (A^+\otimes_{S}A^-)/(A^+\varepsilon_0A^+\otimes_{S_0}A^-\varepsilon_0A^-)\cong A/ A\varepsilon_0A$. To prove the claim  for the base case of the induction, it remains to check that the canonical map 
\begin{equation}
\label{eq:mult}
\gamma: (A^+/A^+\varepsilon_0 A^+) \otimes_{S/S\varepsilon_0 S}  (A^-/A^-\varepsilon_0 A^-) \rightarrow (A^+\otimes_{S}A^-)/(A^+\varepsilon_0A^+\otimes_{S_0}A^-\varepsilon_0A^-)
\end{equation}
which sends $\overline{a}\otimes_{S/S\varepsilon_0 S}\overline{b}$ to $\overline{a\otimes_{S} b}$ is an isomorphism of $k$--vector spaces. 
Let $\pi\colon A^+\rightarrow A^+/A^+ \varepsilon_0A^+$ be the canonical surjection.
Multiplying $\pi$ on the right by any idempotent $\eta$ of degree greater than 0 and comparing degrees shows that
$\ker(\pi)\eta=A^+ \varepsilon_0 A^+\eta=0$, so that $(A^+/A^+\varepsilon_0 A^+)\eta\cong A^+\eta$ as left $A^+$--modules. Similarly one has $\eta (A^-/A^- \varepsilon_0 A^-)\cong \eta A^-$ as right $A^-$--modules. This implies that the left hand side of \eqref{eq:mult} is isomorphic to $\bigoplus_{\deg(\eta)>0}A^+\eta\otimes_{k}\eta A^-$, which in turn is isomorphic to $(\bigoplus_{i=0}^n A^+e_i\otimes_k e_iA^-)/(\bigoplus_{\deg(e_i)=0}  A^+e_i\otimes_k e_iA^-)$ and hence to the right hand side of \eqref{eq:mult}. Thus $\gamma$ is bijective and the induction start is finished.

To continue the induction, we use that in the Reedy algebra $\overline{A}\coloneqq A/J_0$ the two-sided ideal generated by the idempotents of lowest degree is $J_1/J_0$. This allows to proceed exactly as above and to prove that $(A/J_0)/(J_1/J_0)\cong A/J_1$ is Reedy.
\end{proof}

 With the notation introduced before Proposition \ref{prop:more}, we have the following:

\begin{corollary}
\label{cor:heredity chain of Reedy}
A Reedy algebra $(A,A^+,A^-)$ with a degree function $\deg\colon E=\{e_0,e_1,e_2,\dots,e_n\}\rightarrow\mathbb{N}$ admits a heredity chain 
\begin{equation*}
0\subseteq A\varepsilon_0 A\subseteq A(\varepsilon_0+\varepsilon_1 )A \subseteq \cdots \subseteq A(\varepsilon_0 +\varepsilon_1+\cdots +\varepsilon_m   )A=A
\end{equation*}
(here $m\geq \max \{\deg(e_i)\colon i=0, \ldots, n \}$).
The corresponding partial order to this heredity chain (as recalled in Section \ref{sec:two}) satisfies $i \lhd j$ if and only if $\deg(e_i)>\deg(e_j)$. 
In addition, for $e=\varepsilon_0+\varepsilon_1+\cdots + \varepsilon_l$, the centraliser subalgebra $eAe$ is a quasi-hereditary algebra with a heredity chain
$$0\subseteq eJ_0e\subseteq e J_1e \subseteq \cdots \subseteq eJ_le=eAe;$$
and the quotient algebra $A/AeA$ is a quasi-hereditary algebra with a heredity chain 
  $$0\subseteq J_{l+1}/J_{l}\subseteq J_{l+2}/J_l \subseteq \cdots \subseteq
  J_m/J_l=A/J_l.$$
\end{corollary}
\begin{proof}
By Proposition \ref{prop:more}(ii) $A/J_l$ is a Reedy algebra, with a degree function naturally inherited from $A$, and the idempotents of lowest degree in $A/J_l$ generate the ideal $J_{l+1}/J_l$. By Lemma \ref{lem:plus_minus} the ideal $J_{l+1}/J_l$ is a heredity ideal of $A/J_l$, thus $A$ is quasi-hereditary. In particular, the Reedy algebra structure of the algebras $eAe$ and $A/J_l$ found in Proposition \ref{prop:more} implies their quasi-heredity with heredity chains as claimed.
\end{proof}

\begin{remark}
Proposition \ref{prop:more} can alternatively be proved using the equivalence (i)$\Leftrightarrow$(iii) in Theorem \ref{thm:characterization} and results from \cite{CK} (see also \cite[Corollary 5.3]{Koe1}). By Theorem \ref{thm:characterization}, $(A,\unlhd)$ is quasi-hereditary, $A^-$ is an exact Borel subalgebra of $A$ and $A^+$ is a Delta subalgebra of $A$ (that is, $(A^+)^{op}$ is an exact Borel subalgebra of $A^{op}$). Note that $I'\coloneq\{i\in \{0,1,\dots,n\}\colon\deg(e_i)\leq l\}$ is a coideal of $\unlhd$ (meaning that any $j$ satisfying $j\rhd i$ for some $i\in I'$ is contained in $I'$). Observe that $I'$ coincides with $\{i\in \{0,1,\dots,n\}\colon e L^{A^+}(i)\neq 0\}$ and $\{i\in \{0,1,\dots,n\}\colon e L^{A^-}(i)\neq 0\}$. Thus, \cite[Theorem 5.9]{CK} implies that $(eAe,\unlhd)$ is still quasi-hereditary with an exact Borel subalgebra $eA^-e$ and a Delta subalgebra $e A^+ e$. Moreover, $eA^-e$ and $eA^+e$ must be elementary (as $A^-$ and $A^+$ are) and $eA^-e\cap eA^+e=eSe$ is a maximal semisimple subalgebra of both $eA^-e$ and $eA^+e$. It also follows from \cite[Theorem 5.9]{CK} that the canonical ring homomorphisms $\pi^-:A^-/A^-eA^- \to A/AeA$ and $\pi^+:A^+/A^+eA^+ \to A/AeA$ are injective and turn $A^-/A^-eA^-$ and $A^+/A^+eA^+$ into an exact Borel subalgebra and a Delta subalgebra of $(A/AeA, \unlhd)$, respectively. Note that $A^-/A^-eA^-$ and $A^+/A^+eA^+$ are elementary and the pullback of $\pi^-$ and $\pi^+$ is $S/SeS$, which is isomorphic to a maximal semisimple subalgebra of both $A^-/A^-eA^-$ and $A^+/A^+eA^+$. The result now follows from Theorem \ref{thm:characterization}.
\end{remark}

\section{Constructing and characterising Reedy algebras via
  idempotent subalgebras and quotient algebras}

In Theorem~\ref{thm:main} below we characterise Reedyness of $A$ in terms of Reedyness of $A/AeA$ and $eAe$ for certain idempotents $e$ of $A$. This is an analogue for Reedy algebras of Dlab and Ringel's construction and
characterisation of quasi-hereditary algebras along heredity chains, in
\cite[Theorem 1]{DLcomp}. We will make use of the following:

\begin{setup}
\label{setup}
Let $A$ be a finite-dimensional algebra over $k$ with a complete set $E\coloneqq\{e_0,e_1,\dots,e_n\}$ of pairwise orthogonal idempotents and assume that the following data are given:
\begin{itemize}
\item[(i)] a degree function $\deg\colon \{e_0,e_1,\dots,e_n\}\rightarrow \mathbb{N}$,
\item[(ii)] a subalgebra $A^+\subseteq A$ which contains the elements of $E$ and satisfies condition (i) from Definition \ref{def:Reedy},
\item[(iii)] a subalgebra $A^-\subseteq A$ which contains the elements of $E$ and satisfies condition (ii) from Definition \ref{def:Reedy}.
\end{itemize}
In addition, let $\varepsilon_l$ denote the sum of all elements of $E$ having degree $l$ (as before, $\varepsilon_l\coloneqq 0$ if no idempotent in $E$ has degree $l$).
\end{setup}
For such an algebra we consider the chain of ideals
\begin{equation}
\label{eq:candidate_chain}
0\subseteq A\varepsilon_0 A\subseteq A(\varepsilon_0+\varepsilon_1 )A \subseteq \cdots \subseteq A(\varepsilon_0 +\varepsilon_1+\cdots +\varepsilon_m   )A=A.
\end{equation}
Sometimes we write $J_l\coloneqq A(\varepsilon_0 +\varepsilon_1+\cdots +\varepsilon_l )A$ and we also set  $J_{-1}\coloneqq 0$. 

By the assumptions in Setup \ref{setup}, the algebra $S\coloneqq A^+\cap A^-=\prod_{\deg(e_i)=0}^{m}k e_i$ is a common maximal semisimple subalgebra of both $A^+$ and $A^-$, cf.~Lemma~\ref{lem:basic_algebras}.
However, as before it need not be a maximal semisimple subalgebra of $A$ itself. For $l=0,1,\dots,m$ we will also use the notation $S_{l}\coloneqq \prod_{\deg(e_i)=l}k e_i$.

\medskip

The connection between Reedy algebras and quasi-hereditary algebras with triangular decomposition
established in Theorem \ref{thm:characterization} suggests another analogy
that will provide a crucial tool for establishing the second main theorem:

\begin{lemma}
\label{lem:direct_sum_characterizations}
For an algebra $A$ satisfying Setup \ref{setup} the following statements
are equivalent:
\begin{itemize}
\item[(i)] The algebra $(A,A^+,A^-)$ is Reedy.
\item[(ii)] For each $l=0,1,\dots,m$ there is an isomorphism of $k$--vector spaces induced by multiplication
\begin{equation}
\bigoplus\limits_{\deg(e_i)=l} (A^+/A^+\varepsilon_{l-1}A^+)e_i\otimes_k e_i (A^-/A^-\varepsilon_{l-1}A^-) \cong J_l/J_{l-1}.
\nonumber
\end{equation}
\item[(iii)] For each $l=0,1,\dots,m$ there is an isomorphism of $k$--vector spaces induced by multiplication
\begin{equation}
\bigoplus\limits_{\deg(e_i)=l} A^+e_i \otimes_k e_i A^- \cong J_l/J_{l-1}.
\nonumber
\end{equation}
\end{itemize}
\end{lemma}

\begin{proof}
  (i)$\Rightarrow$(ii) By Proposition~\ref{prop:more}(ii), for all $l=0,1,\dots,m$, the algebra $\overline{A}\coloneqq A/J_{l-1}$ is Reedy and $J_{l}/J_{l-1}$ is a heredity ideal at the bottom of a heredity chain of $\overline{A}$.
  As a two-sided ideal, $J_{l}/J_{l-1}$ is generated by the residue classes of the idempotents of $E$ having degree $l$. Thus $J_l/J_{l-1}=\overline{A} \varepsilon_l \overline{A}$. A precise description of the Reedy decomposition of $\overline{A}$ and of the
subalgebras $(\overline{A})^+$ and $(\overline{A})^-$ has been given in
Proposition~\ref{prop:more}(ii).  Combining this with Lemma~\ref{lem:plus_minus}
yields statement (ii).

(ii)$\Rightarrow$(iii)  Let $e$ be any idempotent of degree $l$ and let
$\pi\colon A^+\rightarrow A^+/A^+ \varepsilon_{l-1}A^+$ be the canonical surjection.
Multiplying $\pi$ on the right by $e$ and comparing degrees yields
$\ker(\pi)e=A^+\varepsilon_{l-1}A^+e=0$. Similarly, left multiplication by $e$ of
the canonical surjection $A^-\rightarrow A^-/A^- \varepsilon_{l-1}A^-$ also gives an
isomorphism. Now, statement (iii) follows.

(iii)$\Rightarrow$(i) To be proven is the Reedy decomposition property, i.e.~that multiplication $A^+ \otimes_{S}A^- \to A$ is an isomorphism. The latter map is filtered by the isomorphisms in condition (iii), and thus it is an isomorphism by induction on $l$.
\end{proof}

\begin{theorem}
\label{thm:main}
Let $A$ be an algebra satisfying Setup \ref{setup} and suppose that
$A=A^+\cdot A^-$. Then the following statements are equivalent: 
\begin{itemize}
\item[(i)] The algebra $(A,A^+,A^-)$ is a Reedy algebra.
\item[(ii)] For some idempotent $e=\varepsilon_0+\varepsilon_1+\cdots +\varepsilon_l$ the algebras
\begin{displaymath}
(eAe,eA^+e,eA^-e)\,\,\,\,\, \mbox{and}\,\,\,\,\, (A/AeA,A^+/A^+eA^+,A^-/A^-eA^-)
\end{displaymath}
are Reedy algebras and the multiplication map $Ae\otimes_{eAe}eA\rightarrow AeA$ is bijective.
\item[(iii)] For each idempotent $e=\varepsilon_0+\varepsilon_1+\cdots +\varepsilon_l$ the algebras
\begin{displaymath}
(eAe,eA^+e,eA^-e)\,\,\,\,\, \mbox{and}\,\,\,\,\, (A/AeA,A^+/A^+eA^+,A^-/A^-eA^-)
\end{displaymath}
are Reedy algebras and the multiplication map $Ae\otimes_{eAe}eA\rightarrow AeA$ is bijective.
\end{itemize}
\end{theorem}

\begin{proof}
The implication (iii)$\Rightarrow$(ii) is clear. We now prove (ii)$\Rightarrow$(i). By Lemma~\ref{lem:direct_sum_characterizations} we need to prove that for all $j$ there is an isomorphism 
\begin{equation}
\label{eq:wanted}
\bigoplus\limits_{\deg(e_i)=j} A^+e_i \otimes_k e_i A^- \cong J_j/J_{j-1}.
\end{equation}
We will use the notation $s_\kappa\coloneqq \varepsilon_0+\varepsilon_1+\dots+\varepsilon_\kappa$ for all $\kappa=0,1,\dots,m$. Recall also that $e=s_l=\varepsilon_0+\varepsilon_1+\dots+\varepsilon_l$ is fixed in the statement of (ii). We record a few observations.

Since $eAe$ is Reedy, it has a Reedy decomposition
$eAe\cong \bigoplus_{\deg(e_i)\leqslant l} \,eA^+e_i\, \otimes_{k}\, e_iA^-e$.
Multiplying this decomposition with $s_\kappa$ on both sides, for $s_\kappa$
with $\kappa\leqslant l$, implies that $s_\kappa As_\kappa$ is Reedy as well,
since $s_\kappa A^+ e_i = 0$ for $\deg(e_i) > \kappa$, and analogously for
$A^-$.

Moreover, for all $\kappa\leqslant l$, the (always surjective)
multiplication map 
\begin{equation}
    \label{eq:dagger}
As_\kappa\otimes_{s_\kappa As_\kappa}s_\kappa A\rightarrow As_\kappa A
\end{equation}
is an isomorphism as a restriction of the multiplication map
$Ae\otimes_{eAe}eA\rightarrow AeA$, which has been assumed to be bijective.

In addition, for all $k\leqslant m$, we claim that 
\begin{equation}
\label{eq:trouble}
    A s_{\kappa} /(A s_{\kappa-1} A s_{\kappa})\cong  A^+ s_{\kappa}\quad\text{and} \quad
s_{\kappa} A/ (s_{\kappa} As_{\kappa-1}A )\cong s_{\kappa} A^- 
\end{equation}as right and left $S_k$--modules respectively. Indeed, since $A=A^+\cdot A^-$ by assumption, any element of $A s_{\kappa} /(A s_{\kappa-1}A s_{\kappa})$ is a finite linear combination of equivalence classes of the form 
$c\cdot \eta\cdot b\cdot s_\kappa +A s_{\kappa-1}A s_{\kappa}$, where $c\in A^+$, $\eta\in E$ and $b\in A^-$. By the defining property (iii) of the subalgebra $A^-$ given in Setup \ref{setup}, the only non-zero components in such an expression are those with $\deg(\eta)=\kappa$, in which case $c\cdot \eta\cdot b\cdot s_\kappa=c\cdot \eta\cdot b\cdot\eta \in A^+s_{\kappa}$. This fact allows us to define a homomorphism of right $S_k$--modules from $A s_{\kappa} /(A s_{\kappa-1} A s_{\kappa})$ to $A^+ s_{\kappa}$, which is an isomorphism with inverse isomorphism the canonical map $A^+ s_{\kappa}\rightarrow A s_{\kappa} /(A s_{\kappa-1} A s_{\kappa})$. Similarly, there exists an isomorphism $s_{\kappa} A/ (s_{\kappa} As_{\kappa-1}A )\cong s_{\kappa} A^- $ of left $S_k$--modules. 

We now prove \eqref{eq:wanted} by induction on $\kappa$. By \eqref{eq:trouble}, $A\varepsilon_0=A^+\varepsilon_0$ and
similarly $\varepsilon_0 A=\varepsilon_0A^-$. 
In addition, $\varepsilon_0 A\varepsilon_0\cong S_0$ by the Reedyness of
$\varepsilon_0 A\varepsilon_0$. This concludes the base case $\kappa = 0$
of the induction, by reducing it to the isomorphism in \eqref{eq:dagger}.

For all $0<\kappa \leqslant l$, by \eqref{eq:dagger}, multiplication induces
the following isomorphism 
\begin{equation}
\label{eq:quotients}
\left( A s_\kappa \underset{s_\kappa As_\kappa}{\otimes}s_\kappa A \right) / \left( As_{\kappa-1} \underset{s_{\kappa-1}As_{\kappa-1}}{\otimes} s_{\kappa-1}A \right) \xrightarrow{\cong} As_\kappa A/As_{\kappa-1}A.
\end{equation}
The left hand side of \eqref{eq:quotients} maps onto 
\begin{equation*}
\left(As_\kappa/As_{\kappa-1}As_k\right) \underset{(s_\kappa As_\kappa)/(s_{\kappa}J_{\kappa-1}s_{\kappa})}{\otimes}\left( s_\kappa A/s_k As_{\kappa-1}A\right).
\end{equation*}
The latter vector space maps to $As_\kappa A/As_{\kappa-1}A$ by the rule $\bar{a}\otimes\bar{b}\mapsto \overline{a\otimes b}$, altogether producing a factorisation of the map in \eqref{eq:quotients}. 
Thus the surjection is an isomorphism:
$$
\, \, \left( A s_\kappa \underset{s_\kappa As_\kappa}{\otimes}s_\kappa A \right) / \left( As_{\kappa-1} \underset{s_{\kappa-1}As_{\kappa-1}}{\otimes} s_{\kappa-1}A \right) \cong \left(As_\kappa/As_{\kappa-1}As_k\right) \underset{(s_\kappa As_\kappa)/(s_{\kappa}J_{\kappa-1}s_{\kappa})}{\otimes}\left( s_\kappa A/s_k As_{\kappa-1}A\right).$$
This allows to rewrite  (\ref{eq:quotients}) as 
\begin{equation}\label{eq:quotientsrewritten}
\left(As_\kappa/As_{\kappa-1}As_k\right) \underset{(s_\kappa As_\kappa)/(s_{\kappa}J_{\kappa-1}s_{\kappa})}{\otimes}\left( s_\kappa A/s_k As_{\kappa-1}A\right)\xrightarrow{\cong} As_\kappa A/As_{\kappa-1}A.
\end{equation}

As observed above, the algebra $s_\kappa As_\kappa$ is Reedy, and we are now
going to use this to rewrite the left hand side of
(\ref{eq:quotientsrewritten}). By Reedyness, 
$(s_\kappa As_\kappa)/ (s_{\kappa}J_{\kappa-1}s_{\kappa}) =S_{\kappa}.$ 
Therefore, there is an isomorphism 
\begin{equation*}
\left(As_\kappa/As_{\kappa-1}As_k\right) \underset{S_k}{\otimes}\left( s_\kappa A/s_k As_{\kappa-1}A\right) \xrightarrow{\cong} As_\kappa A/As_{\kappa-1}A.
\end{equation*}
From \eqref{eq:trouble}, there is an isomorphism
\begin{equation*}
  A^+s_\kappa \otimes_{S_\kappa} s_\kappa A^- 
  \xrightarrow{\cong} As_\kappa A/As_{\kappa-1}A.
\end{equation*}
At this point, the 
desired isomorphism \eqref{eq:wanted} has been proved for all
$\kappa=0,1,\dots,l$. 

Moreover, since the algebra $A/J_l$ is Reedy by assumption and
$J_{l+1}/J_l\subseteq \cdots \subseteq J_{m}/J_{l}=A/J_l$ is a heredity chain
of $A/J_l$, Proposition~\ref{prop:more} implies that for all
$\kappa=l+1,\dots,m$ the ideal $J_{\kappa}/J_{\kappa-1}$ is the heredity ideal
generated by all idempotents of lowest degree in the Reedy algebra
$A/J_{\kappa-1}$. Thus to finish the proof, Lemma \ref{lem:plus_minus} can be
employed.

Finally, we prove (i)$\Rightarrow$(iii). The fact that the various algebras $eAe$ and $A/AeA$ are Reedy
is part of Proposition~\ref{prop:more}. It remains to prove that the
multiplication map $Ae\otimes_{eAe}eA\rightarrow AeA$ is an isomorphism of
$k$--vector spaces. 
Using that Reedy algebras are quasi-hereditary, this
follows from \cite[Proposition~7]{DLcomp}.
\end{proof}

\begin{remark}
  The assumption $A=A^+\cdot A^-$ in Theorem~\ref{thm:main} cannot be omitted.
  Indeed, consider the algebra $A\coloneqq \begin{pmatrix}
k & k \\
0 & k
\end{pmatrix}$ and let $e_0=\begin{pmatrix}
1 & 0 \\
0 & 0 
\end{pmatrix}$ and $e_1=\begin{pmatrix}
0 & 0 \\
0 & 1 
\end{pmatrix}$ be idempotents with $\deg(e_0)=0$ and $\deg(e_1)=1$.

Put $S=A^+=A^-=\begin{pmatrix}
k &  0\\
0 & k
\end{pmatrix}$, the $k$-algebra generated by the two idempotents
$e_0$ and $e_1$. Then $A^+\otimes_S A^-=S\subsetneq A$ so that $(A,A^+,A^-)$
is not Reedy. However, the multiplication
map ${Ae_0 \otimes_{k}e_0A\rightarrow Ae_0A}$ is an isomorphism, identifying
the tensor product of the first column, as left module, and the first row,
as right module, with the first row as bimodule. Moreover, the
al\-ge\-bras $e_0Ae_0\cong k$ and $A/Ae_0A\cong k$ have the required Reedy
decompositions. Thus, condition (ii) in Theorem~\ref{thm:main} is
satisfied while condition (i) fails.
\end{remark}

\section*{Acknowledgements} We would like to thank Anna Rodriguez Rasmussen for sharing with us in advance her results on exact Borel subalgebras of quasi-hereditary monomial algebras (\cite{Rasmussen}), as well as for helpful discussions on this topic. We would also like to thank the anonymous referee for a careful reading of the manuscript and for several suggestions that improved the exposition of the paper.

\end{document}